\documentclass[10pt,twoside]{amsart}

\usepackage[initials]{amsrefs}
\usepackage{amsmath,amssymb,amsthm}
\usepackage{tikz-cd}
\usepackage{mathtools}
\usetikzlibrary{cd}

\usepackage{fancyhdr}
\makeatletter
\renewcommand{\section}{%
\@startsection{section}{1}{\z@}{-3.5ex \@plus -1ex \@minus -.2ex}%
   {2.3ex \@plus.2ex}{\normalfont\normalsize\bfseries}}%
\renewcommand{\subsection}{%
\@startsection{subsection}{2}{\z@}{-3.25ex\@plus -1ex \@minus -.2ex}%
   {-1.5ex \@plus.2ex}{\normalfont\normalsize\bfseries}}%
\renewcommand{\@seccntformat}[1]{\csname the#1\endcsname.\quad}
\makeatother
\newtheoremstyle{kjmthm}%
{1em}{1em}{\it}{1em}{\sc}{.}{1em}{\thmname{#1}\ \thmnumber{#2}\thmnote{(#3)}}%
\newtheoremstyle{kjmdef}%
{1em}{1em}{\rm}{1em}{\sc}{.}{1em}{\thmname{#1}\ \thmnumber{#2}\thmnote{(#3)}}%
\renewenvironment{abstract}
{\begin{center}{\bf Abstract}\end{center}\par}{}
\date{}
\newcommand{\kjmfirstpg}{1}%

\makeatletter

\newcommand{\headtitle}{Groups having Wirtinger presentations}
\newcommand{\fullauthor}
{Toshiyuki {\sc Akita} and Sota {\sc Takase}}
\newcommand{\headauhtor}
{Toshiyuki {\sc Akita} and Sota {\sc Takase}}%
\theoremstyle{kjmthm}
\newtheorem{theorem}{Theorem}
\newtheorem{proposition}[theorem]{Proposition}

\newtheorem{corollary}[theorem]{Corollary}

\theoremstyle{kjmdef}

\newtheorem{eg}[theorem]{Example}

\theoremstyle{remark}

\newcommand{\Ker}{\operatorname{Ker}}

\title[Groups having Wirtinger presentations]{Groups having Wirtinger presentations \\
and the second group homology}

\begin{document}
\maketitle

\begin{abstract}
Kuz\cprime min (1996) 
characterized groups having Wirtinger presentations
 in relation to their second group homology. 
 In this paper, we further refine the relation 
 between these groups and their second group homology.
\end{abstract}

\section{Introduction}

A \textit{Wirtinger presentation} is a group presentation 
$\langle X\mid R\rangle$
where all defining relations in $R$ have the following form:
  \[
    x_\beta^{-1}x_\alpha x_\beta=x_\gamma
    \quad (x_\alpha,x_\beta,x_\gamma\in X).
  \]
We call a group having a Wirtinger presentation a \textit{Wirtinger group}.
Wirtinger groups are also called (orientable) {C-groups} \cites{Kulikov,MR1307063,Kuz'min1996} or 
labelled oriented graph (LOG) groups \cites{Howie,MR3627395} 
in the literature.
A Wirtinger group is said to be \textit{irreducible}
if all the generators of its Wirtinger presentation 
are conjugate each other.

Wirtinger groups appear in various contexts of topology and algebra. 
Knot groups and link groups
are well-known examples of Wirtinger groups. 
More generally, for any smooth $n$-dimensional closed orientable manifold
smoothly embedded in $\mathbb{R}^{n+2}$, %
the fundamental group of its complement is known to be a Wirtinger group
\cites{Simon,MR1307063}.
Further examples of Wirtinger groups include
braid groups, pure braid groups, Artin groups,
 Thompson's group $F$,  and associated groups of quandles
\cite{MR3729413}.

After pioneering studies by Yajima \cite{Yajima} and Simon \cite{Simon},
Kuz\cprime min obtained necessary and sufficient conditions for a group 
to be a Wirtinger group \cite{Kuz'min1996}.
Before quoting Kuz\cprime min's results,
we should notice that the abelianization $G_{\mathrm{ab}}$ of a Wirtinger group $G$
is a free abelian group.
In particular, if $G$ is an irreducible Wirtinger group,
then $G_{\mathrm{ab}}$ is an infinite cyclic group.
\begin{theorem}[Kuz\cprime min \cite{Kuz'min1996}*{Theorem 1}]
\label{theorem_Kuz'min1996}
  A group $G$ is a Wirtinger group if and only if
  $G_{\mathrm{ab}}$ is free abelian and
  there exists a subset $B\subset G$ satisfying the following 
  three conditions:
 \begin{enumerate}
 \item $G$ coincides with the normal closure of $B$.
 \item The classes  in $G_{\mathrm{ab}}$ of elements in
 $B$ are linearly independent.
\item 
 The homomorphism
  \[
    \zeta \colon \bigoplus_{a \in B}C_G(a) \to H_2G, \quad
    (g_a)_{a\in B} \mapsto \sum_{a \in B}[a|g_a]-[g_a|a]
  \]
  is surjective, where $C_G(a)$ denotes the centralizer of $a$ in $G$, 
  $H_{2}G$ is the second integral homology group of $G$,
  $[a|g_a]-[g_a|a]$ is a normalized $2$-cycle of $G$
and  $\bigoplus$ is the external direct sum in the category of groups.
 \end{enumerate}  
\end{theorem}%
In case that $G$ is an irreducible Wirtinger group,
a subset $B$ must be a singleton $B=\{a\}$,
and Theorem \ref{theorem_Kuz'min1996} implies
the following result:
\begin{corollary}[Kuz\cprime min \cite{Kuz'min1996}*{Corollary 4}]
\label{cor-kuzmin}
A group $G$ is an irreducible Wirtinger group if and only if
$G_{\mathrm{ab}}\cong\mathbb{Z}$ and there is an element $a\in G$ satisfying
the following condtions:
\begin{enumerate}
\item $G$ coincides with the normal closure of $a\in G$.
\item The homomorphism
  \[
    \zeta \colon C_G(a) \to H_2G, \quad
    g \mapsto [a|g]-[g|a]
  \]
  is surjective.
\end{enumerate}
\end{corollary}

The results of Kuz\cprime min
unveil intrinsic relationships between Wirtinger groups
and the second group homology.
The purpose of this paper is to prove the following theorem
which refines such relationships further:
\begin{theorem}\label{theorem_main}
  Let $G$ be a Wirtinger group and let $B$ be a subset of $G$ 
  which satisfies the conditions in Theorem \ref{theorem_Kuz'min1996}.
  For each 
  $a\in B$, define a homomorphism $\varepsilon_a\colon G \to \mathbb{Z}$ by
  \[
    \varepsilon_a(b) = \begin{cases}
      1 & \text{if $b=a$,}\\
      0 & \text{otherwise}
    \end{cases}\notag
  \]
  for $b\in B$. Then the homomorphism
  \[
    \zeta \colon \bigoplus_{a\in B} C_G(a)\cap\Ker{\varepsilon_a} \to H_2G,
    \quad (g_a)_{a\in B} \mapsto 
    \sum_{a\in B}[a|g_a]-[g_a|a]
  \]
  is surjective.
\end{theorem}
If $G$ is an irreducible Wirtinger group, then $B=\{a\}$ for some $a\in G$
and the homomorphism $\varepsilon_a\colon G\to\mathbb{Z}$
is nothing but the projection to the abelianization $G_{\mathrm{ab}}\cong\mathbb{Z}$.
As an immediate corollary of Theorem \ref{theorem_main},
we obtain the following result:
\begin{corollary}\label{corollary_irreducible case}
Let $G$ be an irreducible Wirtinger group.
Then there exists an element $a\in G$
satisfying
\begin{itemize}
\item[\textup{(i)}] $G$ coincides with the normal closure of $a$,
\item[\textup{(ii)}] the class of $a$ in $G_{\mathrm{ab}}$ generates $G_{\mathrm{ab}}$.
\end{itemize}
Furthermore,
the homomorphism
  \[
    \zeta \colon C_G(a)\cap[G,G] \to H_2G, \quad g \mapsto [a|g]-[g|a]
  \]
  is surjective.
\end{corollary}

\section{The proof of Theorem \ref{theorem_main}}

\begin{proposition}[Brown 
\cite{Brown}*{\S II.5 Exercise 4}]\label{proposition_Brown}
 Let $G=F/R$ be a group where $F$ is a free group
 and $R$ is a normal subgroup of $F$.
Choose a set-theoretical section $s\colon G\to F$,
and define a homomorphism $\varphi\colon H_2G 
  \to {R\cap [F,F]}/[R,F]$ to be the one 
  induced by a homomorphism $C_2G \to R/[R,R], \, [g|h] \mapsto s(g)s(h)s(gh)^{-1}[R,R]$,
  where $C_{2}G$ is the group of normalized $2$-chains of $G$. 
In  addition, define a homomorphism in the opposite direction 
$\psi\colon {R\cap [F,F]}/[R,F]\to H_2G$ by 
  \[
    \prod_{i=1}^{n}[a_i,b_i][R,F] \mapsto \sum_{i=1}^{n}
    \{[I_{i-1}|\bar{a_i}]+[I_{i-1}\bar{a_i}|\bar{b_i}]
    -[I_{i-1}\bar{a_i}\bar{b_i}\bar{a_i}^{-1}|\bar{a_i}]-[I_i|\bar{b_i}]\},
  \]
  where $\bar{a_i}$ is the class of $a_i \in F$ in $G=F/R$ and
  $I_i\coloneqq [\bar{a_1},\bar{b_1}]\cdots[\bar{a_i},\bar{b_i}]$.
Then $\varphi$ is an isomorphism
whose inverse is $\psi$.
\end{proposition}

\begin{proposition}\label{proposition_coincide}
  Under the same notation as Proposition \ref{proposition_Brown},
  let 
 \[
  \zeta'\colon\bigoplus_{a \in B}C_G(a) \to {R\cap [F,F]}/[R,F]
  \]
   be the homomorphism defined by setting
  \[
    \zeta'((g_a)_{a\in B}) = \prod_{a\in B}[s(a), s(g_a)][R,F].
  \]
  Then the composition of $\zeta'$ with the isomorphism $\psi$ in Proposition \ref{proposition_Brown} coincides with the homomorphism
\[
\zeta \colon \bigoplus_{a \in B}C_G(a) \to H_2G, \quad
    (g_a)_{a\in B} \mapsto \sum_{a \in B}[a|g_a]-[g_a|a].
\]
  In particular, 
  the homomorphism $\zeta'$ does not depend on the choice of section $s$.
\end{proposition}
\begin{proof}
  It is enough to show that for each $a$ and $g_a\in C_G(a)$, the equation $\psi\circ\zeta'(g_a)=\zeta(g_a)$ holds. 
  It follows from Proposition \ref{proposition_Brown} that
  \begin{align*}
    \psi\circ\zeta'(g_a)&=\psi([s(a),s(g_a)][R,F]) \\
    &=[1|a]+[a|g_a]-[ag_a a^{-1}|a]-[[a,g_a]|g_a] \\
    &=[a|g_a]-[g_a|a] \\
    &=\zeta(g_a).
  \end{align*}
\end{proof}

\begin{proof}[Proof of Theorem \ref{theorem_main}]
Let $a$ be an arbitrary element of $B$.
Since the right coset 
$\Ker\varepsilon_a\cdot a$ is a generator of an infinite cyclic group 
$\Ker\varepsilon_a\backslash G$, 
for any element $g_a$ of $G$ 
there is a unique pair $(n,g'_a)\in \mathbb{Z}\times\Ker\varepsilon_a$ such that $g_a=g'_a a^n$. 
We take a set-theoretical section $s_a\colon G\to F$ for each $a\in B$ such that 
\[
  s_a(g_a) = s_a(g_a') s_a(a)^n.
\]
A simple calculation yields $[s_a(a),s_a(g_a)]=[s_a(a),s_a(g'_a)]$
and hence $[a,g_a]=[a,g'_a]$, 
from which it follows
that $g_a\in C_G(a)$ if and only if $g'_a\in C_G(a)\cap\Ker\varepsilon_a$.
Therefore, for any $(g_a)_{a\in B} \in \bigoplus_{a\in B} C_G(a)$,
 there is $(g_a')_{a\in B} \in \bigoplus_{a\in B} C_G(a) 
\cap\Ker{\varepsilon_a}$ which satisfies
\[\zeta'((g_a)_{a\in B}) = \zeta'((g_a')_{a\in B}). \]
Finally, the composition with $\psi$ gives rise to $\zeta((g_a)_{a\in B})=\zeta((g'_a)_{a\in B})$. 
This completes the proof of Theorem \ref{theorem_main}.
\end{proof}%

\section{Examples}
Let $G$ be an irreducible Wirtinger group and 
$a\in G$ is an element as in Corollary 
\ref{cor-kuzmin} and \ref{corollary_irreducible case}.
In general, $C_G(a)\cap[G,G]$ 
is a proper subgroup of $C_G(a)$, and therefore, 
Theorem \ref{theorem_main} 
(\emph{resp.} Corollary \ref{corollary_irreducible case})
actually improves upon Theorem \ref{theorem_Kuz'min1996}
(\emph{resp.} Corollary \ref{cor-kuzmin}).
In this section,
we give examples of $G$ and $a\in G$ satisfying
$C_G(a)\cap[G,G]$ is a %
proper subgroup of $C_G(a)$.
\begin{eg}\label{eg-knot}
Let $K\subset S^3$ be a non-trivial tame knot and 
$G_{K}\coloneqq\pi_1(S^3\setminus K)$ be the knot group of $K$.
It is well-known that $G_{K}$ is an irreducible Wirtinger group and that $H_{2}G_{K}=0$.
Let $(\mu,\lambda)$ be a pair of a  meridian $\mu$ 
and a preferred longitude $\lambda$ in $G_{K}$.
Then the meridian $\mu\in G_{K}$ satisfies %
the conditions (i)(ii) Corollary \ref{corollary_irreducible case}. Namely,
\begin{enumerate}
\item $G_{K}$ coincides with the normal closure of $\mu$,
\item the class of $\mu$ in $(G_{K})_{\mathrm{ab}}\cong\mathbb{Z}$ 
generates $(G_{K})_{\mathrm{ab}}$.
\end{enumerate}
Now suppose that $K$ is a prime knot.
Then the annulus theorem (Cannon-Feustel \cite{MR0391094}) implies
the isomorphisms
\[
C_{G_{K}}(\mu)=\langle \mu,\lambda
\rangle\cong\mathbb{Z}\oplus\mathbb{Z},\quad
C_{G_{K}}(\mu)\cap [G_{K},G_{K}]=\langle \lambda\rangle\cong \mathbb{Z},\]
as explained in Eisermann \cite{MR1954330}*{\S4.3}.
\end{eg}

\begin{eg}\label{eg-braid}
Let $B_n$ $(n\geq 4)$ be the braid group on $n$ strands,
and let $\sigma_i$ $(1\leq i\leq n-1)$ be the standard
generaters of $B_n$. Namely,
\[
B_n=\langle \sigma_1,\sigma_2,\dots,\sigma_{n-1}\mid
\sigma_i\sigma_j\sigma_i=\sigma_j\sigma_i\sigma_j\
(|i-j|=1),\ \sigma_i\sigma_j=\sigma_j\sigma_i\
(|i-j|\geq 2)\rangle.
\]
It is known that $B_n$ is an irreducible Wirtinger group
(see Gilbert \cite{MR3627395}*{\S3}), and that
$H_2B_n\cong \mathbb{Z}/2\mathbb{Z}$.
The generator $\sigma_1\in B_n$ satisfies %
the conditions (i)(ii) Corollary \ref{corollary_irreducible case}. 
It follows easily from the presentation of $B_n$ that 
\[
C_{B_n}(\sigma_1)\supset 
\langle\sigma_{1}\rangle\times\langle \sigma_3,\sigma_4,
\dots,\sigma_{n-1}\rangle\cong \mathbb{Z}\times B_{n-2},
\]
and that
\[
\sigma_i\not\in [B_n,B_n]\quad (1\leq i\leq n-1),
\]
where the latter follows from the fact
that each $\sigma_i$ generates the abelianization
$(B_n)_{\mathrm{ab}}\cong\mathbb{Z}$ of $B_n$.
We have verified that $C_{B_n}(\sigma_1)\cap [B_n,B_n]$
is indeed a proper subgroup of $C_{B_n}(\sigma_1)$.
Since the homomorphism 
$\zeta\colon C_{B_n}(\sigma_1)\cap[B_n,B_n]\to H_2B_n$
is surjective and $H_2B_n\cong \mathbb{Z}/2\mathbb{Z}$,
the subgroup $C_{B_n}(\sigma_1)\cap[B_n,B_n]$ is non-trivial.
\end{eg}

\begin{eg}\label{eg-takase}
  Let $G$ be an irreducible Wirtinger group defined by generators $g,x,y,z$ and relations
  \begin{align*}
    &g^{-1}xg=y,\, g^{-1}yg=z,\, g^{-1}zg=x, \\
    &y^{-1}gy=x,\, y^{-1}xy=z,\, y^{-1}zy=g,
  \end{align*}
which was introduced in  Kuz\cprime min \cite{Kuz'min1996}*{Example 2}.
The group $G$ is a semi-direct product of the quaternion group $Q$ %
and the infinite cyclic group generated by $g$.
Kuz\cprime min showed that $[G,G]=Q$ and $H_{2}G=0$ \cite{Kuz'min1996}*{Example 2}.
The generator $g\in G$ satisfies 
the conditions (i)(ii) Corollary \ref{corollary_irreducible case}.
Observe that the order of $C_{G}(g)$ is infinite because 
$C_{G}(g)$ contains the infinite cyclic group generated by $g$,
while the order of $C_G(g)\cap [G,G]$ is finite
because $[G,G]=Q$ is a finite group.
We conclude that 
$C_G(g)\cap [G,G]$ is a finite proper subgroup of $C_G(g)$.

\end{eg}

\section*{Acknowledgements}
The authors thank the anonymous referee for valuable comments that improved the
quality of the manuscript.
The first author was partially supported by JSPS KAKENHI Grant Number
20K03600 and by the Research Institute for Mathematical Sciences, a Joint
Usage/Research Center located in Kyoto University.

\vspace{10mm}
\noindent
\begin{tabular}{l}
Toshiyuki {\sc Akita}\\
Department of Mathematics\\
Hokkaido University\\
Sapporo 060-0810\\
Japan\\
E-mail: akita@math.sci.hokudai.ac.jp
\end{tabular}
\bigskip

\noindent
\begin{tabular}{l}
Sota {\sc Takase}\\
Department of Mathematics\\
Hokkaido University\\
Sapporo 060-0810\\
Japan\\
E-mail: takase.sota.c9@elms.hokudai.ac.jp
\end{tabular}
\bigskip

\label{kjmlast}
\end{document}